%% file: arXiv_template.tex
\documentclass[12pt]{amsart}
\usepackage{mathtools}
\usepackage{amsfonts}
\usepackage{amsthm}
\usepackage{bbm}
\usepackage{enumitem}

\newcommand{\Z}{\mathbb Z}
\newcommand{\R}{\mathbb R}
\newcommand{\N}{\mathbb N}
\newcommand{\F}{\mathcal F}
\newcommand{\1}{\mathbbm 1}
\newcommand{\eps}{\varepsilon}
\newcommand{\dd}{\text{d}}
\newcommand{\cadlag}{c\`{a}dl\`{a}g }
\newcommand{\ccadlag}{c\`{a}dl\`{a}g}
\newcommand{\Levy}{L\'{e}vy }
\renewcommand{\P}{\mathbb P}
\newcommand{\E}{\mathbb E}
\newcommand{\M}{\mathcal M}
\newcommand{\G}{\mathcal G}
\renewcommand{\S}{\mathcal S}
\renewcommand{\o}{o}

\DeclareMathOperator{\isDistr}{\stackrel{d}{=}}

\numberwithin{equation}{section}

\newtheorem{theorem}{Theorem}[section]
\newtheorem{proposition}[theorem]{Proposition}
\newtheorem{lemma}[theorem]{Lemma}
\newtheorem{corollary}[theorem]{Corollary}
\newtheorem{conjecture}[theorem]{Conjecture}
\newtheorem*{claim}{Claim}
\newtheorem*{question}{Question}

\newtheorem*{example}{Example}

\theoremstyle{remark}
\newtheorem{remark}[theorem]{Remark}

\theoremstyle{definition}
\newtheorem{definition}[theorem]{Definition}

\newtheorem{assumption}{Assumption}

\begin{document}

\title[Comparison of partition functions in a random environment]{Comparison of partition functions in a space-time random environment}

\author{Stefan Junk}
\address[Stefan Junk]
{Technical University of Munich, Munich, Germany}
\email{junk@tum.de}


\begin{abstract}
Let $Z^1$ and $Z^2$ be partition functions in the random polymer model in the same environment but driven by different underlying random walks. We give a comparison in concave stochastic order between $Z^1$ and $Z^2$ if one of the random walks has ``more randomness'' than the other. We also treat some related models: The parabolic Anderson model with space-time \Levy noise; Brownian motion among space-time obstacles; and branching random walks in space-time random environments. We also obtain a necessary and sufficient criterion for $Z^1\preceq_{cv}Z^2$ if the lattice is replaced by a regular tree.
\end{abstract}

\maketitle

\input{text}

\input{references.bbl}
\end{document}

%% file: text.tex
\section{A motivating example}\label{sec:intro}

The question discussed in this work arises naturally in the context of the parabolic Anderson model with \Levy noise, and for the purpose of this introduction we further consider the special case of a disastrous Poissonian environment as introduced in \cite{shiga2}. Let $(\omega=\{\omega(t,i)\colon t\in\R_+,i\in\Z^d\},\P)$ be an independent collection of unit intensity Poisson processes and consider the infinite-dimensional system of PDEs
\begin{equation}\label{eq:defPAM}
\begin{split}
u(0,i)&=1\quad\text{ for }i\in\Z^d,\\
\tfrac {\dd }{\dd t}u(t,i)&=\kappa\,(\Delta u)(t,i)-u(t^-,i)\omega(\dd t,i)\quad\text{ for }t\in\R_+,i\in\Z^d
\end{split}
\end{equation}
where $(\Delta u)(t,i)=\tfrac1{2d}\sum_{|i-j|=1}(u(t,j)-u(t,i))$ is the discrete Laplacian. We give an interpretation for the dynamics: Initially every site has mass one. Whenever the environment $\omega(\cdot,i)$ at site $i$ has a jump at time $t$ all mass currently at site $i$ is removed from the system (that is, the mass at site $i$ is set to zero). On the other hand the Laplacian $\Delta$ spreads the mass at $i$ to the neighboring sites at rate $\kappa$. 

We should think of the environment $\omega$ as a random disorder which is smoothed by the Laplacian. The parameter $\kappa$ adjusts the strength of this smoothing and we expect that (in some sense) the solution to \eqref{eq:defPAM} is less random if $\kappa$ is larger. Indeed this is the conclusion of our main result.

Let us introduce a more convenient representation for the solution of \eqref{eq:defPAM}: We can identify $\omega$ with the set of its jump times, i.e. we regard $\omega\subseteq\R_+\times\Z^d$ with $(t,i)\in\omega$ if $\omega(t,i)=\omega(t^-,i)+1$. Following the intuition from before we think of an element $(t,i)\in\omega$ as a disaster at time $t$ at site $i$. Let $(X=\{X(t)\colon t\geq 0\},P^\kappa)$ be a simple random walk on $\Z^d$ with jump rate $\kappa\geq 0$ started in the origin, independent of $\omega$. The graph of $X$ almost surely intersects $\omega$, and we consider the (quenched) survival probability 
\begin{align}\label{eq:defquenchedsurv}
Z_t^\kappa(\omega)\coloneqq P^\kappa\big((s,X(s))\notin\omega\text{ for all }s\in[0,t])\big).
\end{align}
Since we interpret $\omega$ as random space-time disasters this is the probability that the random walk has not been killed up to time $t$. We stress that $Z$ is a random variable depending on $\omega$, and it has been shown \cite[Lemma 2.1]{shiga} that $u(t,0)\isDistr Z^\kappa_t$ for every $t>0$. 

We can use this representation to get an intuition for how there is ``more randomness'' in the solution to \eqref{eq:defPAM} at small $\kappa$: Since $X(t)=0$ for all $t\in\R_+$ under $P^0$,
\begin{align}\label{eq:example_0}
Z^0_t(\omega)=\1\{\omega\cap ([0,t]\times\{0\})=\emptyset\}.
\end{align}
On the other hand for large jump rates $X$ can potentially visit a large part of the environment before time $t$, so from the spatial ergodicity of $\omega$ one expects 
\begin{align}\label{eq:example_infty}
Z^\kappa_t(\omega)\approx\E[Z_t^\kappa(\omega)]=e^{-t}\quad\text{ for }\kappa\gg 0.
\end{align}
At the qualitative level \eqref{eq:example_0} depends on the environment while \eqref{eq:example_infty} is (almost) deterministic. In Theorem \ref{thm:main} we will see that this observation holds more generally and not only for the extreme cases $\kappa=0$ and $\kappa\approx\infty$. More precisely we show that $\kappa\mapsto Z^\kappa_t$ is increasing in concave stochastic order, i.e. for every $t>0$, $\kappa_1\leq\kappa_2$ and all concave, real functions $f$
\begin{align}\label{eq:toy}
\E[f(Z_t^{\kappa_1})]\leq\E[f(Z^{\kappa_2}_t)].
\end{align}

For an intuition note that the survival probability is small if the environment is such that several disasters form a trap close to the origin, forcing the random walk to behave atypically to avoid it. Having a high jump rate is helpful because it allows the random walk more flexibility to go around problematic areas. However \eqref{eq:example_0} shows that a high jump rate is not always optimal since there are environments with no disasters at the origin. In that sense we can view a low jump rate as an all-or-nothing gamble where the survival probability is large if there are no disasters close to the origin, and very small otherwise. This is an indication that comparison in the concave order $\preceq_{cv}$ is the correct notion since it is a measure on the ``riskiness'' of two random variables. 


The contribution of this work is to show that the implication
\begin{align}\label{eq:meta}
\text{``more randomness in $X$''}\quad\implies\quad \text{``less randomness in $Z$''}
\end{align}

is a universal property of such models; and to establish a framework for how such models can be treated in a general way regardless of their specific features.

This framework will be introduced in the next section, in which we keep the disastrous parabolic Anderson model as an ongoing example. The proof of the main result follows in Section \ref{sec:main}. We then discuss applications of Theorem \ref{thm:main} in a number of models: We will revisit the parabolic Anderson model with general \Levy noise in Section \ref{sec:ctds}; the random polymer model is obtained by replacing continuous time $\R_+$ with discrete time $\N$ (Section \ref{sec:dtds}); replacing instead discrete space $\Z^d$ by $\R^d$ gives a continuous-space version of the parabolic Anderson model where for simplicity we only discuss the case of a Poissonian environment (Section \ref{sec:ctcs}). In addition there is a close connection between the random polymer model and branching random walks in a random space-time environment, and here the conclusions from Sections \ref{sec:dtds}-\ref{sec:ctcs} can be applied to prove an interesting phase transition in the survival probability (Sections \ref{sec:dt_branching} and \ref{sec:ct_branching}). We also discuss the relevant literature for each model.

Note that the parabolic Anderson model is more commonly studied with $\omega$ replaced by a random field which is static in time or has long-time correlations. This situation is not covered by our assumptions (where the environment has independent increments), and we do not expect that the implication \eqref{eq:meta} holds. We shortly discuss this in Section \ref{sec:static}.

Moreover recall that continuous-time random walks have a convolution property, i.e. if $X$ and $X'$ are independent random walks of jump rate $\kappa$ and $\kappa'$ then $X+X'$ is a random walk of jump rate $\kappa+\kappa'$. We use this fact in the proof of Theorem \ref{thm:main} but we expect that one can also obtain the comparison between partition functions under weaker assumptions. In Section \ref{sec:conj} we make a conjecture for the optimal criterion and in Section \ref{sec:trees} we show that it is indeed necessary and sufficient if the lattice is replaced by a tree.

\section{Notation}\label{sec:setup}

We now want to generalize to a wider class of models in discrete or continuous time and in more general state spaces. Let $T$ be equal to $\N$ or $\R_+$ and let $S$ be a (not necessarily commutative) group. We use ``$+$'' for the group action and $0$ for the neutral element. All the examples below will be on the lattice ($S=\Z^d$) or in Euclidean space ($S=\R^d$). Let $\Sigma$ denote the set of \cadlag paths $x\colon T\to S$ (In discrete time $T=\N$ we drop the requirement of ``\ccadlag''). For $x,y\in\Sigma$ we use $x+y$ to denote the path obtained by coordinate-wise addition
\begin{align*}
(x+y)_t\coloneqq x_t+y_t\quad\text{ for }t\in T.
\end{align*}
By a slight abuse of notation we also use $0$ for the trivial path with $x_t=0$ for all $t\in T$. 

Let $\G$ be the smallest sigma field such that all projections $t\mapsto x_t$ are $\G$-measurable. Moreover let $(\Omega,\F)$ be a probability space and let $\{\theta^x\colon x\in\Sigma\}$ be a family of $\F$-measurable bijections $\theta^x:\Omega\to\Omega$. An element $\omega\in\Omega$ is called an environment while $\theta^x$ is called the shift associated to the path $x$.

\begin{example}
We use the disastrous parabolic Anderson model as an ongoing example: Let $T=\R_+$ and $S=\Z^d$. We choose $\Omega$ the set of locally finite subsets of $\R_+\times\Z^d$ (i.e. $|\omega\cap A|<\infty$ for all $A\subseteq\R_+\times\Z^d$ compact). For $x\in\Sigma$ and $\omega\in\Omega$ the shifted environment $\theta^x(\omega)$ is obtained by moving all disasters in $\omega$ according to the displacement of $x$. That is,
\begin{align}\label{eq:defshift}
(t,i)\in\omega\quad\iff\quad (t,i-x_t)\in(\theta^x(\omega)).
\end{align}
\end{example}

\begin{definition}
Let $F\colon\Omega\times\Sigma\to\R_+$ be $\F\otimes \G$-measurable. We say that $F$ is \textbf{consistent} if $F(\omega,x+y)=F(\theta^y(\omega),x)$ for every $\omega\in\Omega$, $x,y\in\Sigma$. 
\end{definition}
\begin{definition}
Let $F\colon\Omega\times\Sigma\to\R_+$ be a consistent function. For $P\in\M(\Sigma)$ we call
\begin{align}\label{eq:realdefz}
Z(\omega)\coloneqq \int_{\Sigma}F(\omega,x)P(\dd x).
\end{align}
the \textbf{partition function} of $P$.
\end{definition}

Let now $\P$ be a probability measure on $(\Omega,\F)$ and write $\E$ for its expectation. It is important to distinguish it from the law $P$ and expectation $E$ of the random walk. Note that $Z(\omega)=E[F(\omega,X)]$ is a random variable on $(\Omega,\mathcal F,\P)$.

\begin{example}[Continued]
For $\omega\in\Omega$, $x\in\Sigma$ let
\begin{align}\label{eq:defFdis}
F_t(\omega,x)\coloneqq \1\{ (s,x_s)\notin\omega\text{ for all }s\in[0,t]\}
\end{align}
the indicator of the event that $x$ survives until time $t$. Note that $F_t$ is consistent: For $\omega\in\Omega$, $x,y\in\Sigma$ 
\begin{align*}
F_t(\omega,x+y)=1&\iff (s,x_s+y_s)\notin\omega\text{ for }s\in[0,t]\\&\iff (s,y_s)\notin (\theta^x(\omega))\text{ for }s\in[0,t]\iff F_t(\theta^x(\omega),y)=1.
\end{align*}
Let now $\P$ denote the Poisson point process with unit intensity on $\R_+\times\Z^d$ and $P^\kappa$ the law of continuous time simple random walk with jump rate $\kappa$. Then the partition function $Z^\kappa_t$ of $P^\kappa$ is the quenched survival probability from \eqref{eq:defquenchedsurv}.
\end{example}

We make the following assumptions:

\begin{assumption}[Shift invariance]\label{ass:spatial}
$\P(A)=\P\big((\theta^x)^{-1}(A)\big)$ for all $A\in\F$, $x\in\Sigma$.
\end{assumption}

\begin{assumption}[Integrability]\label{ass:inte}
$\E[F(\omega,0)]<\infty$.
\end{assumption}

Under these assumptions the partition function is a well-defined random variable. 

\begin{lemma}\label{lem:annealed}
Let $Z$ be the partition function of $P\in\M(\Sigma)$ and $\P$ such that \eqref{ass:spatial} and \eqref{ass:inte} hold. Then $Z$ is $\F$-measurable and $\E[Z]<\infty$. In particular $\P$-almost surely $Z<\infty$. Moreover if $Z'$ is the partition function of $P'\in\M(\Sigma)$ then $\E[Z]=\E[Z']$.
\end{lemma}

This is an easy consequence of Fubini's Theorem and the consistency of $F$, so we skip the proof. We have now finished setting up the model. Next we need to introduce an order for probability measures on $\Sigma$ to make the notion ``$X$ has more randomness'' from \eqref{eq:meta} precise. Note that since $\Sigma$ inherits the group structure from $S$ we can define a convolution on $\Sigma$:
\begin{definition}
Let $P,Q\in\M(\Sigma)$. Then $P*Q$ is the law of $X+Y$ where $X$ and $Y$ are independent and have laws $P$ and $Q$.
\end{definition}
This defines an ordering $\preceq_*$ on $\M(\Sigma)$:
\begin{definition}\label{def:part_conv}
For $P^1,P^2\in\mathcal M(\Sigma)$ we write $P^2\preceq_*P^1$ if $P^2=P^1*Q$ for some $Q\in\M(\Sigma)$.
\end{definition}

\begin{example}[Continued]
In the disastrous parabolic Anderson model assumption \eqref{ass:inte} is clear from $F_t\leq 1$, and \eqref{ass:spatial} follows because $\P$ is a homogeneous Poisson point process on $\R_+\times\Z^d$, i.e. spatially invariant. The annealed partition function is indeed independent of $\kappa$:
\begin{align*}
\E[Z^\kappa_t]=\E[F_t(\omega,0)]=\P\big(|\omega\cap([0,t]\times\{0\})|=0\big)=e^{-t}
\end{align*}
Moreover $P^{\kappa}*P^{\kappa'}=P^{\kappa+\kappa'}$ and therefore $P^{\kappa_2}\preceq_*P^{\kappa_1}$ whenever $\kappa_1\leq\kappa_2$.
\end{example}
\begin{remark}
It is somewhat counter-intuitive to use $P^2\preceq_*P^1$ if $P^2$ is more random than $P^1$, but we adopt this notation for consistency with the majorization order $\preceq_M$ that we discuss in Section \ref{sec:outlook}.
\end{remark}

\begin{remark}
 Note that $\preceq_*$ is a reflexive and transitive relation, i.e. it defines a pre-order on $\M(\Sigma)$. If $S$ is commutative we can extend $\preceq_*$ to a partial order by identifying $P^1$ and $P^2$ if there exists a deterministic $z\in\Sigma$ such that $P^1$ is the law of $P^2(\cdot+z)$. Note that in this case $Z^1\isDistr Z^2$ when $Z^1$ resp. $Z^2$ is the partition function of $P^1$ and $P^2$, which follows from \eqref{ass:spatial}
\end{remark}

\section{The main result}\label{sec:main}

\begin{theorem}\label{thm:main}
Let \eqref{ass:spatial} and \eqref{ass:inte} be satisfied and $P^1,P^2\in\M(\Sigma)$ with $Z^1$ resp. $Z^2$ the partition function of $P^1$ resp. $P^2$. Assume that $P^2\preceq_*P^1$. There exists a coupling $(\widehat Z^1,\widehat Z^2)$ such that $\widehat Z^i\isDistr Z^i$ for $i=1,2$, and almost surely
\begin{align}\label{eq:strassen}
\widehat Z^2=\E\big[\widehat Z^1\big|\widehat Z^2\,\big].
\end{align}
\end{theorem}

\begin{proof}
From Definition \ref{def:part_conv} we find $Q\in\M(\Sigma)$ such that 
\begin{align}\label{eq:conv}
P^2=P^1*Q.
\end{align}
Let $\omega$ and $Y$ be independent with laws $\P$ and $Q$, define $\widetilde \omega\coloneqq \theta^Y(\omega)$ and 
\begin{align*}
\widehat Z^1&\coloneqq \int_{\Sigma}F(\widetilde\omega,x)P^1(\dd x)\\
\widehat Z^2&\coloneqq \int_{\Sigma\times\Sigma}F(\omega,x+y)P^1(\dd x) Q(\dd y).
\end{align*}
From \eqref{eq:conv} it is clear that $\widehat Z^2\isDistr Z^2$. Moreover $Y$ and $\omega$ are independent so \eqref{ass:spatial} yields $\widetilde \omega\isDistr \omega$. We therefore have $\widehat Z^1\isDistr Z^1$ as well, and setting $\mathcal A\coloneqq \sigma(\omega)$ we compute
\begin{align*}
\E\big[\widehat Z^1\big|\mathcal A\big]&=\int_{\Sigma}\int_{\Sigma}F(\theta^y(\omega),x)P^1(\dd x)Q(\dd y)\\
&=\int_{\Sigma}\int_{\Sigma}F(\omega,x+y)P^1(\dd x)Q(\dd y)=\widehat Z^2.
\end{align*}
We have used the consistency of $F$ in the second equality. Since $\widehat Z^2$ is $\mathcal A$-measurable \eqref{eq:strassen} follows from the tower property of conditional expectation.
\end{proof}

\begin{corollary}
Under the assumptions from Theorem \ref{thm:main} we have $Z^1\preceq_{cv}Z^2$, i.e. for all $f\colon\R_+\to[-\infty,\infty)$ concave 
\begin{align}\label{eq:strassen2}
\E\big[f(Z^1)\big]\leq \E\big[f(Z^2)\big].
\end{align}
\end{corollary}
\begin{proof}
First note that the expectations on both sides are well-defined in $[-\infty,\infty)$ by the concavity of $f$ and \eqref{ass:inte}. There is nothing to do if the LHS equals $-\infty$, so in the following we  can assume $f(Z^1)\in L^1$. Using the coupling from Theorem \ref{thm:main}
\begin{align*}
\E\big[f(Z^2)\big]=\E\big[f\big(\widehat Z^2\big)\big]&=\E\big[f\big(\E[\widehat Z^1|\widehat Z^2]\big)\big]\\
&\geq \E\big[\E\big[f\big(\widehat Z^1\big)\big|\widehat Z^2\big]\big]=\E\big[f\big(\widehat Z^1\big)\big]=\E\big[f(Z^1)\big].
\end{align*}
The inequality is due to Jensen's inequality for conditional expectations. The second-to-last equality follows from the tower-property for conditional expectations, using the assumption $f(Z^1)\in L^1$.
\end{proof}

\begin{remark}
The concave ordering $Z^1\preceq_{cv}Z^2$ is equivalent to the existence of a coupling satisfying \eqref{eq:strassen}, a result known as Strassen's Theorem for the concave ordering, see \cite[Theorem 1.5.20]{mullerstoyan}. We point out that in contrast to the increasing stochastic order $\preceq_{st}$ the coupling for $\preceq_{cv}$ is typically not explicitly known. (Recall that for the increasing stochastic order on $\R$ the coupling is realized simply by plugging a uniform random variable in the respective quantile functions).
\end{remark}

\section{Applications}\label{sec:applications}

\subsection{The random polymer model}\label{sec:dtds}
We present a model in discrete and discrete space, so let $T\coloneqq \N$ and $S\coloneqq \Z^d$. The environment consists of real random variables associated to each space-time point, so let $\Omega\coloneqq [-1,\infty)^{T\times S}$ and $\P$ such that $\omega$ is i.i.d. satisfying
\begin{align}\label{eq:expmom}
R\coloneqq \E[(1+\omega(0,0))]<\infty.
\end{align}
Let $\Sigma$ denote the set of all paths $\N\to\Z^d$ and for $\omega\in\Omega$, $x\in\Sigma$ define the shifted environment $\theta^x(\omega)$ by
\begin{align*}
(\theta^x(\omega))(t,i)\coloneqq \omega(t,i+x_t).
\end{align*}
That is, $\theta^x$ acts on $\omega$ by shifting the environment in each ``time-slice'' according to the corresponding displacement of $x$. We have to specify a consistent function $F_t\colon\Omega\times\Sigma\to\R_+$. For $\omega\in\Omega$, $x\in\Sigma$ and $t\in\N$ let 
\begin{equation}\label{eq:dtF}
\begin{split}
H_t(\omega,x)&\coloneqq \sum_{s=1}^t\log\big(1+\omega(s,x_s)\big)\1\{\omega(s,x_s)>-1\}\\
G_t(\omega,x)&\coloneqq \1\{\omega(s,x_s)>-1\text{ for all }s=1,\dots,t\}\\
F_t(\omega,x)&\coloneqq e^{H_t(\omega,x)}G_t(\omega,x).
\end{split}
\end{equation}
Note that with the convention ``$e^{\log (0)}=0$'' we could simplify the definition and write
\begin{align*}
F_t(\omega,x)=e^{\sum_{s=1}^t \log(1+\omega(s,x_s))}.
\end{align*}
Also note that $G_t(\omega,x)$ is the event that $x$ avoids all sites where $\omega$ takes the value $-1$, which we can interpret as hard-obstacles for the path. We check that $F$ is consistent: Let $I_{s}(\omega,x)\coloneqq \1\{\omega(s,x_s)>-1\}$ and note that $I_s(\omega,x+y)=I_s(\theta^y(\omega),x)$. Then indeed
\begin{align*}
G_t(\omega,x+y)&=\prod_{s=1}^tI_s(\omega,x+y)=\prod_{s=1}^tI_s(\theta^y(\omega),x)=G_t(\theta^y(\omega),x)\\
H_t(\omega,x+y)&=\sum_{s=1}^t\log\big(1+\omega(s,x_s+y_s)\big)I_s(\omega,x+y)\\
&=\sum_{s=1}^t\log\big(1+(\theta^y(\omega))(s,x)\big)I_s(\theta^y(\omega),x)=H_t(\theta^y(\omega),x)
\end{align*}
Assumption \eqref{ass:spatial} follows because $\P$ is i.i.d. and \eqref{ass:inte} follows from \eqref{eq:expmom}. Note that if $P$ is the law of simple random walk on $\Z^d$ then the partition function of $P$ agrees with the partition function in the random polymer model at temperature one. This model has been studied extensively and we refer to \cite{comets_survey}, \cite{yoshida_survey} and \cite{hollander_survey} for surveys.

\begin{remark}
We point out that the assumptions can be weakened: It is enough if $\P$ is independent in time and stationary in space, i.e. $\omega(t,\cdot)$ and $\omega(s,\cdot )$ are independent for $s\neq t$ and for $t\in\N$, $i\in S$ 
\begin{align*}
\{\omega(t,j)\colon j\in S\}\isDistr \{\omega(t,i+j)\colon j\in S\}.
\end{align*}
We do not expect that the conclusion of Theorem \ref{thm:main} remains valid without independence in time (see Section \ref{sec:static}).
\end{remark}
In discrete time the order $\preceq_*$ is less intuitive than in continuous time. A natural example might be random walks with binomially distributed increments:
\begin{example}
Let $X$ resp. $Y$ be such that $X_{t+1}-X_t\sim\operatorname{Bin}(a_t,p_t)$ and $Y_{t+1}-Y_t\sim\operatorname{Bin}(b_t,p_t)$ for all $t\in\N$, for $a_t,b_t\in\N$ and $p_t\in[0,1]$. Then $Y\preceq_*X$ holds if $a_t\leq b_t$ for all $t\in\N$.
\end{example}

We obtain the following consequence from Theorem \ref{thm:main}:
\begin{corollary}\label{cor:dpre}
Let $P^1,P^2\in\M(\Sigma)$ be such that $P^2\preceq_*P^1$, and let $Z^1_t$ resp. $Z^2_t$ denote the partition functions of $P^1$ resp. $P^2$. Then for all $f$ concave 
\begin{align}\label{eq:dpre_res}
\E\big[f(Z^1_t)\big]\leq \E\big[f(Z^2_t)\big]
\end{align}
\end{corollary}

Note that until now we have not excluded the disastrous case where $\P(\omega(0,0)=-1)>0$, so that potentially $\P(Z_t=0)>0$. For the rest of this section we make the more conventional assumption 
\begin{align}\label{eq:nondeg}
\E\big[e^{|\log(1+\omega(0,0))|}\,\big]<\infty.
\end{align}
In that case it is known \cite[Proposition 1.5]{yoshida_path} that there exists $\lambda_0^i\in\R$ (called the (quenched) free energy) such that almost surely
\begin{align}\label{eq:deffree}
\lambda_0^i=\lim_{t\to\infty}\tfrac 1t\E\big[\log Z^i_t\big]=\lim_{t\to\infty}\tfrac 1t\log Z^i_t
\end{align}
From \eqref{eq:dpre_res} it is immediate that $P^2\preceq_*P^1$ implies
\begin{align}\label{eq:firstcomp}
\lambda_0^1\leq \lambda_0^2.
\end{align}
We obtain a second consequence of Corollary \ref{cor:dpre} by considering the martingales $\{W^1_t=Z^1_te^{-\alpha t}\colon t\in\N\}$ and $\{W^2_t=Z^2_te^{-\alpha t}\colon t\in\N\}$, where 
\begin{align*}
\alpha\coloneqq \log \E\big[1+\omega(0,0)\big].
\end{align*}
Since $W^1$ and $W^2$ are non-negative the almost sure limits $W^1_\infty$ and $W^2_\infty$ exist, and by a standard application of the 0-1 law for $i=1,2$
\begin{align}\label{eq:dicht}
\P(W^i_\infty>0)\in\{0,1\}.
\end{align}
From this it is clear that $W^i$ converges in $L^1$ if and only if $\P(W_\infty^i>0)=1$. We are going to exploit this dichotomy to show that the martingale limits are monotone in the following sense:
\begin{align}\label{eq:dpre_res2}
W^1\text{ converges in }L^1\quad\implies\quad W^2\text{ converges in }L^1.
\end{align}
\begin{proof}[Proof of \eqref{eq:dpre_res2}]
Consider the fractional moments of $W^i$, i.e. the supermartingales
\begin{align*}
\big\{(W^i_t)^{1/2}\colon t\in\N\big\}.
\end{align*}
Since both processes are $L^2$-bounded they are uniformly integrable, so that
\begin{align*}
\E\big[(W^1_\infty)^{1/2}\,\big]=\lim_{t\to\infty}\E\big[(W^1_t)^{1/2}\,\big]\leq \lim_{t\to\infty}\E\big[(W^2_t)^{1/2}\,\big]=\E\big[(W^2_\infty)^{1/2}\,\big].
\end{align*}
The claim follows from Corollary \ref{cor:dpre} together with \eqref{eq:dicht}.
\end{proof}

To the best of our knowledge \eqref{eq:firstcomp} and \eqref{eq:dpre_res2} are the first results in this direction, possibly because in discrete time there is no natural parameter that corresponds to the jump rate in continuous time. At present the only work on stochastic ordering in the context of the random polymer model seems to be \cite{nguyen}. However in that work the underlying random walk is kept fixed and the comparison is between environments at different temperatures.

\subsection{The parabolic Anderson model among \Levy noise}\label{sec:ctds}

We have already discussed a special case of this model as our ongoing example in Sections \ref{sec:intro} and \ref{sec:setup}. Recall that we have considered the case where $\omega(\cdot,i)$ is a Poisson process with jumps of size $-1$. In this section we discuss \eqref{eq:defPAM} for more general environments. 

Let $T=\R^+$, $S=\Z^d$ and $\Sigma$ the set of right-continuous paths having finitely many jumps in every compact interval. Let $\Omega$ be the set of $\omega:T\times S\to \R$ such that for every $i\in S$ the mapping $t\mapsto \omega(t,i)$ is \cadlag with jump sizes bounded from below by $-1$, i.e. such that
\begin{align*}
\omega(i,t)\geq \omega(i,t^-)-1\quad\text{ for all }t\in T.
\end{align*}

Let $\F$ be the smallest sigma-field such that all projections $t\mapsto \omega(t,i)$ are measurable and $\P$ such that $\omega$ is a collection of independent \Levy processes. More precisely let $\{B(\cdot,i)\colon i\in S\}$ be a family of independent standard Brownian motions and $\{\eta(\cdot,i,\cdot)\colon i\in S\}$ an independent family of i.i.d. Poisson point processes on $T\times [-1,\infty)$, and set
\begin{align}\label{eq:defdefomega}
\omega(t,i)\coloneqq -\tfrac {\sigma^2}2t+\sigma B(t,i)+\int_{[0,t]\times[-1,\infty)}r\,\eta(\dd s,i,\dd r).
\end{align}
We assume that the intensity measure $\dd s\rho(\dd r)$ of $\eta(\cdot,i,\cdot)$ has finite mass and satisfies
\begin{align}\label{eq:firstmoment}
R\coloneqq\int_{[-1,\infty)}(1+r)\rho(\dd r)<\infty.
\end{align}
The general form of the parabolic Anderson problem with \Levy noise is given by
\begin{equation}\label{eq:pam}
\begin{split}
u(0,i)&= 1\quad\text{ for all }i\in S\\
\tfrac{\dd}{\dd t}u(t,i)&=(Au)(t,i)+u(t^-,i)\omega(\dd t,i)\quad\text{ for all }t\in T,i\in S,
\end{split}
\end{equation}
where $A$ is bounded and is the generator of a Markov process on $\Z^d$. Note that the Brownian motion only contributes to $\omega$ if $\sigma^2\neq 0$, so we can recover the disastrous case from Section \ref{sec:intro} by setting $\sigma=0$ and choosing $\delta_{-1}$ for the intensity measure $\rho$ of $\eta$.

Let $P$ be the law of the process corresponding to $A$. As a first step we use the so-called Feynman-Kac representation for the solutions of \eqref{eq:pam} to define a function $F_t\colon \Omega\times \Sigma\to\R_+$ such that the partition function of $P$ has the same distribution as $u(t,0)$. For $\omega\in\Omega$, $x\in\Sigma$, $t>0$ let
\begin{align}\label{eq:newdefF}
F_t(\omega,x)\coloneqq e^{H_t(\omega,x)+H'_t(\omega,x)}G_t(\omega,x),
\end{align}
where
\begin{align*}
H_t(\omega,x)\coloneqq &\sum_{i\in S}\int_0^t\1\{x_s=i\}B(\dd s,i)\\
H'_t(\omega,x)\coloneqq &\sum_{i\in S}\int_{[0,t]\times (-1,\infty)}\log(1+r)\1\{x_s=i\}\eta(\dd s,i,\dd r)\\
G_t(\omega,x)\coloneqq &\1\Big\{\sum_{i\in S}\int_{[0,t]\times\{-1\}}\1\{x_s=i\}\eta(\dd s,i,\dd r)=0\Big\}.
\end{align*}
Note that this is slightly different from the definition of $F$ in discrete time in the previous section: In the exponent in \eqref{eq:dtF} we take the sum of the environment $\omega$ along the path $x$, while here we instead integrate the increments of the environment $\omega$ along the path $x$. Let us give some intuition for the dynamics. If the path $x$ occupies site $i$ in $[s_1,s_2]$ we get the following contributions in $F_t(\omega,x)$:
\begin{itemize}[noitemsep]
 \item From the Brownian motion we get $e^{B(s_2,i)-B(s_1,i)}$. To maximize $F_t(\omega,x)$ we therefore need to consider paths $x$ that mostly observe regions where the environment is increasing. 
 \item If $|\eta([s_1,s_2],i,\{-1\})|>0$ then $G_t(\omega,x)=0$ and this path does not contribute to $F_t(\omega,x)$. We think of $(s,-1)\in\eta(\cdot,i,\cdot)$ as a hard obstacle at time $s$ at site $i$ and $G_t$ is the indicator function of the event that $x$ survives.
 \item Similarly we think of $(s,r)\in \eta([s_1,s_2],i,(-1,0))$ as a soft obstacle occupying $i$ at time $s$ where the process survives with probability $(1+r)\in(0,1)$. 
 \item The event $(s,r)\in\eta([s_1,s_2],i,(0,\infty))$ can be interpreted as a bonus occupying site $i$ at time $s$ that will result in a contribution $(1+r)>1$. We can maximize $F_t$ by finding paths $x$ that collect many such bonuses.
\end{itemize}

Finally for $Z_t$ the contribution $F_t(\omega,x)$ is integrated against the entropic cost of the path $x$ under $P$. It was shown that \eqref{eq:pam} has an almost surely unique solution which agrees in distribution with $Z_t$. More precisely let $\overleftarrow \omega$ be the time-reversal of $\omega$, i.e. for $s\in[0,t]$ 
\begin{align*}
\overleftarrow \omega(s,i)\coloneqq \omega(t)-\omega((t-s)^-).
\end{align*}
Then \cite[Theorems 2.1 and 2.2]{shiga2} equation \eqref{eq:pam} has a path-wise unique strong solution $u=u(t,i)$ and we have the Feynman-Kac representation
\begin{align*}
Z_t(\overleftarrow\omega)=u(t,0)
\end{align*}
Note that we have added a linear term in \eqref{eq:defdefomega} to simplify the formula for $F$. We finish the definition of the model by specifying shifts $\{\theta^x\colon x\in\Sigma\}$: Similar to the previous section $\theta^x$ acts on $\omega$ by shifting each time-slice according to the displacement of the random walk. More precisely for $t\in T,i\in S$
\begin{align*}
(\theta^x(\omega))(t,i)\coloneqq \sum_{j\in S}\int_0^t\1\{j=i+x_s\}\omega(\dd s,j).
\end{align*}
We write $\theta^{x}(B)$ resp. $\theta^x(\eta)$ for the Brownian motion resp. the jump part of the shifted environment. 

We check that $H_t$ is consistent: For $\omega\in\Omega$, $t>0$ and $x,y\in\Sigma$
\begin{align*}
H_t(\omega,x+y)=&\sum_{j\in S}\int_0^t \1\{x_s+y_s=j\}B(\dd s,j)\\
=&\sum_{i,j\in S}\int_0^t \1\{x_s=i\}\1\{i+y_s=j\}B(\dd s,j)\\
=&\sum_{i\in S}\int_0^t \1\{x_s=i\}(\theta^y(B))(\dd s,i)=H_t(\theta^y(\omega),x)
\end{align*}
A similar calculation for $H_t'$ and $G_t$ shows that $F_t$ is consistent. Moreover for \eqref{ass:inte} we compute using \eqref{eq:firstmoment}
\begin{align*}
\E[F_t(\omega,0)]=\E\big[e^{-\frac{\sigma^2}2t+B(t,0)}\big]\E\big[R^{|\eta([0,t],0,[-1,\infty))|}\big]<\infty.
\end{align*}

Finally since $\omega(\cdot,i)$ has independent increments and since $\P$ is spatially homogeneous it is clear that $(\theta^x(\omega))(\cdot,i)\isDistr \omega(\cdot,i)$ for all $i\in S$. Note that $(\theta^x(\omega))(\cdot,i)$ is a function of the increments of $\omega$ in 
\begin{align*}
g^x(i)\coloneqq \{(t,x_t+i)\colon t\in T\}\subseteq T\times S
\end{align*}
the graph of $x$ shifted by $i$. Clearly $g^x(i)\cap g^x(j)=\emptyset$ for $i\neq j$, and this together with the independent increments of $\omega$ implies that $(\theta^x(\omega))(\cdot,i)$ and $(\theta^x(\omega))(\cdot,j)$ are independent. These considerations show that \eqref{ass:spatial} is satisfied. 

\begin{remark}
Note that in contrast to \cite{shiga2} we have chosen to present the model only in the case where the \Levy measure is finite, i.e. $\omega(\cdot,i)$ has only finitely many discontinuities in every compact interval. This is to keep the notation simple by avoiding the use of the compensated jump measure. Corollary \ref{cor:ctds} below also holds in the general case.
\end{remark}

\begin{remark}
Similar to the situation in discrete time we can consider slightly more general settings: For Corollary \ref{cor:ctds} it is enough to assume that both $\omega$ and $X$ are processes with independent increments, not necessarily stationary in time.
\end{remark}

In the following we focus on the simple random walk case where $A$ is the discrete Laplacian $\kappa\Delta$, which acts on functions $f\colon \Z^d\to\R$ by
\begin{align*}
(\Delta f)(i)=\frac{1}{2d}\sum_{j\colon |i-j|=1}\big(f(j)-f(i)\big).
\end{align*}
Let $P^\kappa$ denote the law of the Markov process corresponding to $\kappa\Delta$ and write $Z^\kappa_t$ for the partition function of $P^\kappa$. In Section \ref{sec:setup} we have already observed that $P^\kappa$ is decreasing in $\preceq_*$, i.e. $\kappa_1\leq\kappa_2$ implies $P^{\kappa_2}\preceq_*P^{\kappa_1}$. Thus from Theorem \ref{thm:main} we obtain the following
\begin{corollary}\label{cor:ctds}
For every $t>0$, $\kappa_1\leq\kappa_2$ and $f$ concave
\begin{align*}
\E[f(Z^{\kappa_1}_t)]\leq \E[f(Z^{\kappa_2}_t)].
\end{align*}
\end{corollary}

We discuss some consequences: It is known \cite{shiga}, \cite{shiga2}, \cite{cranston} that there exists a value $\lambda_0(\kappa)\in\R$ (called the quenched Lyapunov exponent) such that almost surely and in $L^1$
\begin{align}\label{eq:deflambda}
\lim_{t\to\infty}\tfrac 1t\E[\log Z_t(\omega)]=\lim_{t\to\infty}\tfrac 1t\log Z_t=\lambda_0(\kappa).
\end{align}
We point out that in contrast to the discrete-time model in the previous section the limit is defined also in the hard obstacle case $\rho(\{-1\})>0$. We also consider the following quantity, known as the $r^{th}$-annealed Lyapunov exponent:
\begin{align}\label{eq:deflambdap}
\lim_{t\to\infty}\tfrac 1t\log \E\big[(Z_t^\kappa)^r\big]^{\frac 1r}=:\lambda_r(\kappa).
\end{align}
Existence of this limit follows from the sub-/superadditive Lemma. As a direct consequence of Corollary \ref{cor:ctds} 
\begin{align}\label{eq:consequence}
\kappa\mapsto \lambda_r(\kappa)\quad\left\{\begin{matrix}\text{ is increasing if }r< 1\\\text{ is decreasing if }r>1\end{matrix}\right..
\end{align}
Recall from Lemma \ref{lem:annealed} that $\kappa\mapsto\lambda_1(\kappa)$ is constant. Moreover as in the previous section one can consider the martingales $W^\kappa_t\coloneqq  Z_t^\kappa e^{-\alpha t}$ where $\alpha=\log\E[Z_1^\kappa]$, and the same argument as before shows that for $\kappa_1\leq\kappa_2$ 
\begin{align*}
W^{\kappa_1}_t\text{ converges in }L^1\quad\implies\quad W^{\kappa_2}_t\text{ converges in }L^1.
\end{align*}

\subsection{Poissonian environments in continuous space}\label{sec:ctcs}

In this section we discuss the continuous-space version of the previous model. Recall that a component of the environment we considered an i.i.d. collection $\{B(\cdot,i)\colon i\in\Z^d\}$ of Brownian motions indexed by the sites of $\Z^d$. The natural generalization to continuous space is to consider space-time white noise, which however introduces considerable technical difficulties. In the following we will therefore restrict ourselves to Poissonian environments. We refer to \cite{carmona2} for a discussion of space-time white noise environments. 

In the Poissonian case we can encode the environment by a countable collection of triples $(s,i,r)\in\R_+\times\R^d\times[-1,\infty)$, which we interpret as an obstacle/bonus of strength $1+r$ that is present at time $s$ at site $i$. 

More precisely let $T=\R^+$, $S=\R^d$ and $\Omega$ the set of locally finite point measures on $\R_+\times\R^d\times [-1,\infty)$. As before we will identify an element $\omega\in\Omega$ with its support, i.e. we regard $\omega\subseteq\R_+\times\R^d\times[-1,\infty)$. Let $\P$ be the law of a Poisson point process on $\Omega$ with intensity measure $\dd s\,\dd x\,\rho(\dd r)$ where $\rho$ has finite mass and satisfies
\begin{align}\label{eq:finfinitemom}
R:=\int_{[-1,\infty)} (1+r)\rho(\dd r)<\infty.
\end{align}

Similar to the previous section we set
\begin{align*}
H_t(\omega,x)&\coloneqq \int_{[0,t]\times\R^d\times (-1,\infty)} \1\{i\in B_1(x_s)\}\log(1+r) \omega(\dd s,\dd i,\dd r)\\
G_t(\omega,x)&\coloneqq \1\Big\{\int_{[0,t]\times\R^d\times \{-1\}}\1\{i\in B_1(x_s)\}\omega(\dd s,\dd i,\dd r)=0\Big\}\\
F_t(\omega,x)&\coloneqq e^{H_t(\omega,x)}G_t(\omega,x).
\end{align*}
Here $B_1(x)$ denotes the ball of radius $1$ around $x\in\R^d$. The interpretation of $F_t$ is similar to before: If $(s,i,-1)\in\omega$ then we interpret this as a hard obstacle at time $s$ at site $i$, and the process is killed if it is within distance $1$ of $i$ at time $s$. The event $(s,i,r)\in\omega$ can be interpreted as a soft obstacle (for $r\in(-1,0)$) resp. a bonus (for $r>0$) occupying $i\in\R^d$ at time $s$, which the process should avoid resp. should try to collect. 

We define the shifted environment $\theta^x(\omega)$ by the relation
\begin{align*}
(s,i,r)\in\omega\quad\iff\quad (s,i-x_s,r)\in \theta^x(\omega).
\end{align*}
The proof that $F_t$ is consistent is similar to the example in Section \ref{sec:setup}. Moreover \eqref{ass:spatial} directly follows from the spatial homogeneity of $\P$ and \eqref{ass:inte} can be check in the same way as in Section \ref{sec:ctds}.
\begin{remark}
Again we point out that the assumptions can be relaxed, since it is enough that both environment and random walk have independent increments and are stationary in space.
\end{remark}

Let now $P^{\sigma^2}$ be the law of Brownian motion with variance $\sigma^2$ and write $Z^{\sigma^2}_t$ for the partition function of $P^{\sigma^2}$. It is well-known that  $P^{\sigma_1^2}*P^{\sigma_2^2}=P^{\sigma_1^2+\sigma_2^2}$, so that $\sigma^2\mapsto P^{\sigma^2}$ is decreasing in $\preceq_*$. Applying Theorem \ref{thm:main} we therefore get the following
\begin{corollary}\label{cor:ctcs}
For $t>0$, $\sigma^2_1\leq\sigma^2_2$ and $f$ concave
\begin{align*}
\E\big[f\big(Z^{\sigma^2_1}_t\big)\big]\leq \E\big[f\big(Z^{\sigma^2_2}_t\big)\big].
\end{align*}
\end{corollary}
Let us point out that $Z^{\sigma^2}$ has an integrability issue if $\rho(\{-1\})>0$ since in this case $\E[\log Z^{\sigma^2}_t]=-\infty$ for all $t>0$ (see \cite[Proposition 1.2]{fukushima}). We therefore first consider the case $\rho=\delta_{\alpha}$ for $\alpha>-1$, and we write $\P^\alpha$ for the law of this environment. In \cite[Theorem 2.2.1]{comets2004some} it was shown that there exists $\lambda_0(\alpha, \sigma^2)\in(-\infty,\infty)$ such that $\P^\alpha$-almost surely
\begin{align}\label{eq:defbrownianfree}
\lim_{t\to\infty}\frac 1t\E^\alpha\big[\log Z_t^{\sigma^2}\big]=\lim_{t\to\infty}\frac 1t\log Z_t^{\sigma^2}=\lambda_0(\alpha, \sigma^2).
\end{align}
From Corollary \ref{cor:ctcs} for every $\alpha>-1$
\begin{align}
\sigma\mapsto \lambda_0(\alpha,\sigma^2)\quad\text{ is increasing.}
\end{align}

The conclusion is also valid in a disastrous environment $\P^{-1}$: For this we consider a truncated version $F^1_t$ of $F_t$ where $G_t$ has been replaced by 
\begin{align*}
G_t^1(\omega,x)\coloneqq \1\Big\{\int_{[1,t]\times \R^d\times\{-1\}}\1\{i\in B_1(x_s)\}\omega(\dd s,\dd i,\dd r)=0\Big\}.
\end{align*}
In words, $G_t^1$ is the indicator of the event that the Brownian motion avoids all hard obstacles $(s,i)$ with $s\geq 1$.
It has been shown \cite[Theorem 1.3]{fukushima} that this truncation solves the integrability issue but does not affect the almost sure limit in \eqref{eq:defbrownianfree}, i.e. there exists $\lambda_0(-1,\sigma^2)\in\R$ such that $\P^{-1}$-almost surely
\begin{align*}
\lim_{t\to\infty}\frac 1t\E^{-1}\big[\log Z_t^{\sigma^2,1}\big]=\lim_{t\to\infty}\frac 1t\log Z_t^{\sigma^2}=p_0(-1, \sigma^2).
\end{align*}
The same calculation as before shows that $F_t^1$ defines an integrable, consistent function so that $\sigma\mapsto p(-1,\sigma^2)$ is increasing by Corollary \ref{cor:ctcs}.

\subsection{Branching random walks in discrete time}\label{sec:dt_branching}

The random polymer model and its continuous-time relatives have a natural connection to branching processes in a space-time random environment. In this section we show that Theorem \ref{thm:main} can be applied to prove a phase transition for this model, see Corollary \ref{cor:dt_br}. Let $T=\N$, $S=\Z^d$ and $\Sigma$ the set of paths $x\colon T\to S$. Moreover let $\Omega$ be the set of $\eta=\{\eta(t,i)\colon t\in T,i\in S\}$ such that $\eta(t,i)\in\M(\N)$ for every $t\in T,i\in S$. For $\eta\in\Omega$ we define $\omega=\omega(\eta)\in\R^{T\times S}$ by the relation
\begin{align}\label{eq:defomega}
\omega(t,i)\coloneqq \sum_{k\in\N} k\eta(t,i)(k)-1.
\end{align}
In words, an element $\eta\in\Omega$ defines an offspring distribution for every space-time site and $\omega(\eta)+1$ is the expected number of descendants. We assume 
\begin{align*}
-\infty <\E\big[\log (1+\omega(0,0))\big]\leq \log \E\big[(1+\omega(0,0))\big]&<\infty
\end{align*}

Now let $p\in\M(S)$ be an increment distribution and let $(\{X(t)\colon t\in\N\},P^p)$ denote the corresponding random walk. That is, $\{X(t+1)-X(t)\colon t\in\N\}$ is i.i.d. with distribution $p$ under $P^p$. It is clear that $p\mapsto P^p$ is increasing in $\preceq_*$, i.e. 
\begin{align*}
p\preceq_*q\quad\implies\quad P^p\preceq_*P^q.
\end{align*}
Given $\eta\in\Omega$ we consider a branching process $A=\{A(t,i)\colon t\in T,i\in S\}$ with values in $\N^{T\times S}$ that is informally defined in the following way:
\begin{itemize}[noitemsep]
 \item At time $t=0$ there is only one particle at the origin, i.e. $A(0,i)=\1\{i=0\}$.
 \item Suppose the process has been defined until time time $t$ and consider a particle that in generation $t$ occupies site $i$:
 \begin{itemize}[noitemsep]
  \item For generation $t+1$ this particle is replaced by a random number of descendants which are sampled according to $\eta(t,i)$. 
  \item Each descendant then independently moves from $i$ to a random location $i+D$ where the displacement $D$ has distribution $p$. 
 \end{itemize}
 \item This procedure is applied independently for each particle in generation $t$, and we let $A(t+1,i)$ denote the number of particles that occupy $i$ in generation $t+1$.
\end{itemize}
We use $P^p_\eta$ to denote the law of $A$ for a fixed realization of $\eta$, and $\P^p$ for the joint law of $\eta$ and $A$. Let $\{A\text{ survives}\}$ denote the event that for every $t\in\N$ there is $i\in S$ such that $A(t,i)>0$. In this section we discuss how the survival probability depends on the displacement $p$.

For this we let $\widetilde F_t\colon [-1,\infty)^{S\times T}\times \Sigma\to\R_+$ denote the function defined in \eqref{eq:dtF} which we extend to a function $F_t\colon \Omega\times \Sigma$ by $F_t(\eta,x)\coloneqq \widetilde F_t(\omega(\eta),x)$. Let $Z^p_t$ denote the partition function of $P^p$ associated with $F_t$. It was shown (\cite[Lemma 1.4]{yoshida_CLT}) that for every $\eta\in\Omega$
\begin{align}\label{eq:many}
E^p_\eta\Big[\sum_{i\in S}A(t,i)\Big]=Z^p_t(\eta).
\end{align}
Note that on the LHS we have the expected number of particles in the branching random walk in environment $\eta$, while the RHS is the partition function of a random walk in the environment $\omega(\eta)$ induced by $\eta$. To analyse this, recall that almost surely for $t\to\infty$ 
\begin{align*}
Z^p_t\approx e^{t \lambda_0(p)}
\end{align*}
where $\lambda_0(p)=\lim_{t\to\infty}\tfrac 1t\log Z^p_t$ is the free energy associated to $Z_t^p$, see \eqref{eq:deffree}. It is thus intuitively clear that we can expect very different behavior depending on the sign of $\lambda_0(p)$. It turns out that $\lambda_0(p)$ is sufficient to characterize the survival of $A$. More precisely it was shown (\cite[Theorem 2.1.1]{yoshida} and \cite[Theorem 1]{garetmarchand}) that
\begin{align}\label{eq:char}
\P^p(A\text{ survives})>0\quad\iff \quad \lambda_0(p)>0.
\end{align}
From \eqref{eq:firstcomp} we get the following 
\begin{corollary}\label{cor:dt_br}
Assume $p\preceq_*q$ and $\P^q(A\text{ survives})>0$. Then \\$P^p(A\text{ survives})>0$.
\end{corollary}

\begin{remark}
Note that we cannot expect the stronger conclusion 
\begin{align*}
``\P^q(A\text{ survives})\leq \P^p(A\text{ survives})``.
\end{align*}
Consider two classical Galton-Watson processes $(A_n)_{n\in\N}$ and $(A'_n)_{n\in\N}$, each with constant deterministic offspring distribution. The comparison \eqref{eq:firstcomp} is then simply a first moment comparison $\E[A_1]\leq\E[A_1']$. It is well-known that $A$ (resp. $A'$) has positive survival probability if and only if $\E[A_1]>0$ (resp. $\E[A_1']>0$) so we can conclude that $\P(A\text{ survives})>0$ implies $\P(A'\text{ survives})>0$. There are, however, examples where $\E[A_1]<\E[A'_1]$ but $\P(A\text{ survives})>\P(A'\text{ survives})$. 
\end{remark}

\subsection{Branching random walks among disasters}\label{sec:ct_branching}

In this setting we discuss a continuous time-version of a branching random walk in space-time random environment. We choose the same environment as in the disastrous parabolic Anderson model from Section \ref{sec:setup}, i.e. we let $\Omega$ denote the set of locally finite point measures on $\R_+\times\Z^d$. As before we identify $\omega\in\Omega$ with its support and interpret $(t,i)\in\omega$ as a disaster present at time $t$ at site $i$. We consider a process $A=\{A(t,i)\colon t\in\R_+,i\in\Z^d\}$ that evolves according to the following rules:
\begin{itemize}[noitemsep]
 \item Initially we have one particle occupying the origin, $A(0,i)=\1\{i=0\}$.
 \item Each particle independently moves as a simple random walk with jump rate $\kappa>0$.
 \item Each particle independently branches at rate $\lambda>0$, i.e. at rate $\lambda>0$ the particle dies and is replaced by two descendants at the same site.
 \item A disaster $(t,i)\in\omega$ kills all particles that occupy site $i$ at time $t$.
\end{itemize}
This model has the jump rate $\kappa$ and the branching rate $\lambda$ as parameters, and we write $\P^{\kappa,\lambda}$ for the joint law of environment $\omega$ and branching random walk. As before $P^\kappa$ denotes the law of simple random walk with jump rate $\kappa$, and $Z^\kappa_t$ is the partition function of $P^\kappa$ corresponding to $F_t$ from \eqref{eq:defFdis}. Due to \cite[Lemma 1]{self} the following many-to-one formula holds for each $\omega\in\Omega$
\begin{align*}
E^\kappa_\omega\Big[\sum_{i\in S}A(t,i)\Big]=e^{\lambda t}Z^\kappa_t(\omega).
\end{align*}
Recall from \eqref{eq:deflambda} that $Z^\kappa_t$ has a deterministic exponential decay rate $\lambda_0(\kappa)$, so similar to the previous section we expect different behavior depending on the sign of $\lambda_0(\kappa)+\lambda$. Indeed it was shown \cite[Theorem 1.1]{self} that
\begin{align*}
\P^{\kappa,\lambda}(A\text{ survives})>0\quad\iff \quad \lambda+\lambda_0(\kappa)>0.
\end{align*}
In Section \ref{sec:ctds} we have seen that $\kappa\mapsto \lambda_0(\kappa)$ is increasing:
\begin{corollary}
Assume that $\P^{\kappa_1,\lambda}(A\text{ survives})>0$ and $\kappa_1\leq\kappa_2$. Then $\P^{\kappa_2,\lambda}(A\text{ survives})>0$.
\end{corollary}

\section{Outlook}\label{sec:outlook}

In this section we discuss two possible ways in which Theorem \ref{thm:main} might be generalized: 
\begin{enumerate}[noitemsep]
 \item[(a)] One can try to weaken the assumption that the environment has independent increments (which is implicit in \eqref{ass:spatial}). This is particularly interesting for the parabolic Anderson model which is often studied in the case where the environment has long-time correlations.
 \item[(b)] One can also try to weaken the relation $\preceq_*$, so that in Theorem \ref{thm:main} we obtain $Z^1\preceq_{cv}Z^2$ for more partition functions. Here it is particularly interesting to consider the discrete-time setting where $\preceq_*$ is a rather unnatural condition.
\end{enumerate}
We will keep the discussion of (a) short since in general one cannot expect the conclusion of Theorem \ref{thm:main} to hold in this case (Section \ref{sec:static}). For (b) we recall the notion $\preceq_M$ of majorization which we show is a natural candidate for this extension. We conjecture that under some additional assumption it is the optimal condition for Theorem \ref{thm:main} (Section \ref{sec:conj}). As evidence for this conjecture we then prove that $\preceq_M$ really is optimal if we consider random walks on trees instead of on the lattice (Section \ref{sec:trees}).

\subsection{Environments with long-time correlations}\label{sec:static}

We again consider the parabolic Anderson model 
\begin{align*}
u(0,i)&=0\quad\text{ for all }i\in\Z^d\\
\tfrac{\dd }{\dd t}u(t,i)&=\kappa\,(\Delta u)(t,i)+u(t,i)\omega(\dd t,i)\quad\text{ for all }t\in\R_+,i\in\Z^d
\end{align*}
with environment $\omega=\{\omega(t,i)\colon t\in\R_+,i\in\Z^d\}$, but this time we do not assume that $\omega$ has independent increments. Examples include:
\begin{enumerate}[noitemsep]
 \item Static environment: Let $\{\xi(i)\colon i\in\Z^d\}$ be i.i.d. real random variables satisfying some integrability conditions, and set $\omega(t,i)\coloneqq\xi(i)t$ for $t\in\R_+, i\in\Z^d$. 
 \item Independent simple random walks: We consider a field $\{\eta(t,i)\colon t\in\R_+,i\in\Z^d\}$ where $\eta(t,i)$ counts the number of particles occupying site $i$. Assume that the initial numbers $\eta(0,i)$ are independent and Poisson distributed, and that afterwards all particle independently move as simple random walks. We set $\omega(\dd t,i)\coloneqq\beta \eta(t,i)\dd t$ with $\beta$ some (positive or negative) parameter.
 \item Simple exclusion process/Voter model: We consider a field $\{\eta(t,i)\colon t\in\R_+,i\in\Z^d\}$ with values in $\{0,1\}$ where $\eta(0,\cdot)$ is sampled according to a product measure of Bernoulli distributions and where the field afterwards evolves according to the simple exclusion process/the voter model. We again define $\omega(\dd t,i)\coloneqq \beta \eta(t,i)\dd t$ for $\beta$ some (positive or negative) parameter.
\end{enumerate}
Note that we have only presented these examples in an informal way, and we refer to \cite{gaertnerquenched}, \cite{erhardparabolic}, \cite{drewitzsurvival} as well as the survey \cite{koenigparabolic} for a precise definition and an overview of known results. We point out that in these models it is already an interesting problem to study the annealed partition function $\E[Z_t^\kappa]$ while in our setup the annealed partition function does not depend on $\kappa$ (recall Lemma \ref{lem:annealed}). 

Intuitively for these models we can expect that the environment will exhibit islands where it takes significantly larger values than the average environment and which are persistent over a long time. Let us for example consider a static environment taking two values $a<b\in\R$, i.e. with $\P(\xi(0)=a)=1-\P(\xi(0)=b)=p\in(0,1)$. Then (in dimension one) we can expect that among the sites $[-\sqrt t,\sqrt t]\cap\Z$ there exists some interval $I$ whose length is of order $\log (t)$ with $\xi(i)=b$ for all $i\in I$. Now both
\begin{itemize}[noitemsep]
 \item the probability of moving from the origin to the center of $I$ before time $\eps t$ and 
 \item the probability of not leaving $I$ in $[\eps t,t]$ 
\end{itemize}
are sub-exponential. We therefore expect that $u(t,0)\approx e^{bt}$ for large $t$. Note that this corresponds to a qualitatively different strategy than before: The optimal strategy is to immediately move to an area where environment is good and then never again leave this area. It is intuitively clear that for the second part of the strategy above a small jump rate $\kappa$ is better, and we expect $Z^\kappa_t$ to be (in some sense) decreasing in $\kappa$, at least for $t$ large. 

We expect the same phenomenon for the other models: For example consider an environment consisting of independent random walks which have repulsive interaction with the random walk (Example 2 with $\beta<0$). This is the model considered in \cite{drewitzsurvival}, where it was shown \cite[Proposition 2.1]{drewitzsurvival} that 
\begin{align*}
\E[Z^\kappa_t]\leq \E[Z^0_t]\quad \text{ for all }\kappa>0.
\end{align*}
In the physics literature this phenomenon is known under the name ``Pascal's principle'', see the discussion in \cite{drewitzsurvival}. On the heuristic level we think that the concave order is not the correct criterion: Recall that in the disastrous environment from the introduction the quenched survival probability could take very small values if a trap was present in the environment. A high jump rate was thus beneficial because it provides a ``hedge'' against such a scenario by spreading the mass over a large part of the environment. On the other hand in a static environment we can hope for a large value in the partition function $Z_t^\kappa$ if we find a good environment taking only large values close to the origin. In that situation having a small jump rate is beneficial, because we can take full advantage of the benefit without being force to move away from it. This is an indication that in the static case the correct notion is to compare $\E[f(Z^\kappa_t)]$ for different jump rate when $f$ is convex (i.e. risk-seeking) instead of concave (i.e. risk-averse).

It is an interesting question for future research to investigate the transition between static environments and the \Levy-type environments covered by Theorem \ref{thm:main}. More precisely for $\kappa_1\leq\kappa_2$ one might conjecture that if the environment has long-time correlations (as in Examples 1-3) then $Z_t^{\kappa_2}\preceq_{cx}Z^{\kappa_1}_t$ while $Z^{\kappa_1}_t\preceq_{cv}Z^{\kappa_2}_t$ holds if the environment has correlations that decay fast in time.

\subsection{Weakening the convolution property}\label{sec:conj}

In this section we only discuss the one-dimensional random polymer model, so let $T=\N$, $S=\Z$, $\Omega\coloneqq [-1,\infty)^{T\times S}$ and $F_t$ as in \eqref{eq:dtF}. For $p\in\M(S)$ we write $P^p$ for the laws of a random walk whose increments have distribution $p$, and we write $Z^p_t$ for the partition functions of $P^p$ corresponding to $F_t$.

Let now $p,q\in\M(S)$ be compactly supported increment distributions, and recall that in Theorem \ref{thm:main} we have seen that $Z^q_t\preceq_{cv}Z^p_t$ holds if $p\preceq_*q$. Before we have said that $p\preceq_*q$ means $p$ has more randomness than $q$. This implies that under $P^p$ the mass is spread out more evenly than under $P^q$, i.e. it is distributed over a larger set of paths. In this section we ask if this is enough. More precisely:
\begin{question}
Under what conditions on $p$ and $q$ do we have $Z^q_t\preceq_{cv}Z^p_t$ for all $\P$ satisfying \eqref{ass:spatial} and \eqref{ass:inte}?
\end{question}
We start by proving a necessary condition for $p$ and $q$. Recall that for simplicity we have assumed that $p$ and $q$ are compactly supported, so let $K$ be large enough that $p(\{-K,\dots,K\})=q(\{-K,\dots,K\})=1$. Let moreover $\pi$ and $\sigma$ be bijections $\{1,\dots,2K+1\}\to\{-K,\dots,K\}$ chosen in such a way that the functions $i\mapsto p(\pi(i))$ and $i\mapsto q(\sigma(i))$ are decreasing. In other words, $\pi$ and $\sigma$ encode the relative orders of the weights of $p$ and $q$. We say ``$p$ is majorized by $q$'' (written $p\preceq_{M}q$) if 
\begin{align}\label{eq:defmajo}
\sum_{i=1}^k p(\pi(i))\leq \sum_{i=1}^kq(\sigma(i))\quad\text{ for all } k=1,\dots,2K+1.
\end{align}
More generally for $p,q\in [0,\infty)^{\{-K,\dots,K\}}$ (not necessarily with total mass one) we say $p\preceq_M q$ if in addition to \eqref{eq:defmajo} we have
\begin{align*}
\sum_{a=-K}^Kp(a)=\sum_{a=-K}^Kq(a).
\end{align*}
Intuitively $p\preceq_M q$ means that the mass of $q$ is distributed in a more uneven fashion. Observe that the minimal element with respect to $\preceq_M$ is the uniform distribution on $\{-K,\dots,K\}$ (where the mass is spread evenly among all sites) whereas the maximal objects are the Dirac measures (where all mass concentrates on one site). The relation $\preceq_M$ is for example used to compare the distribution of wealth within societies, where it is related to the famous Gini-coefficient. We refer to \cite{olkin} for a survey of results about the majorization order. We observe that  $p\preceq_Mq$ is a necessary condition:
\begin{proposition}\label{prop:necc}
Assume $Z_t^q\preceq_{cv}Z^p_t$ for all $\P$ satisfying \eqref{ass:spatial} and \eqref{ass:inte}. Then $p\preceq_Mq$.
\end{proposition}
\begin{proof}
We construct a suitable $\P$: For $r\in\{-K,\dots,K\}$ define $\omega^{(r)}\in\Omega$ by
\begin{align*}
\omega^{(r)}(t,i)=\left\{\begin{matrix}0&\text{ if }t\geq 2\text{ or }(t=1\text{ and }i-r\equiv_K 0)\\-1&\text{ else. }\end{matrix}\right.
\end{align*}
Here $\equiv_K$ denotes equivalence in $\Z/_{\{-K,\dots,K\}}$, the torus of size $2K+1$. In words, at time $t=1$ there are hard obstacles on $\{-K,\dots,K\}\setminus\{r\}$ and the environment is trivial everywhere else. Let $R$ be uniformly distributed on $\{-K,\dots,K\}$ and let $\P$ be the law of $\omega^{(R)}$. Since $\omega$ is spatially invariant and bounded it is clear that $\P$ satisfies \eqref{ass:spatial} and \eqref{ass:inte}, so for every concave function $f$ 
\begin{align*}
(2K+1)\E[f(Z^p_t)]&=\sum_{r=-K}^K f(Z^p_t(\omega^{(r)}))=\sum_{r=-K}^Kf(p(r))\\&\leq \sum_{r=-K}^Kf(q(r))=\sum_{r=-K}^K f(Z^q_t(\omega^{(r)}))=(2K+1)\E[f(Z^q_t)].
\end{align*}
It is known that this is equivalent to $p\preceq_Mq$, see \cite[Theorem 3.C.1]{olkin}.
\end{proof}

Let us also point out that the relation $\preceq_*$ from Theorem \ref{thm:main} is stronger than $\preceq_M$, i.e. $p\preceq_*q$ implies $p\preceq_Mq$. This follows from the previous proposition together with Theorem \ref{thm:main}, or it can be checked directly (see \cite[Proposition 12.N.1]{olkin}).

However $\preceq_M$ is not a sufficient criterion on its own: Note that if $p$ is obtained from $q$ by permuting the weights in $\{-K,\dots.,K\}$ then $p\preceq_Mq$ and $q\preceq_Mp$. So if $\preceq_M$ was sufficient we would get $Z^p_t\preceq_{cv}Z^q_t\preceq_{cv}Z^p_t$, hence $Z^p_t\isDistr Z^q_t$ for some non-trivial choices of $\P$. But on the lattice this is clearly not true for $t>1$. Let us therefore place some restriction on the ordering of the weights: 
\begin{assumption}\label{ass:uni}
Both $p$ and $q$ are symmetric and unimodal, i.e. the functions $i\mapsto p(i)$ and $i\mapsto q(i)$ are symmetric and decreasing for $i\geq 0$.
\end{assumption}

We conjecture that this is enough:
\begin{conjecture}\label{conj}
Assume \eqref{ass:uni}. Then $p\preceq_M q$ if and only if $Z^q_t\preceq_{cv} Z^p_t$ for all $\P$ satisfying \eqref{ass:spatial} and \eqref{ass:inte}.
\end{conjecture}

We close by mentioning a model where we expect monotonicity but Theorem \ref{thm:main} does not apply: In \cite{discretedisasters} and \cite{nakajima} the authors consider an i.i.d. Bernoulli environment of hard obstacles, i.e. 
\begin{align*}
\P(\omega(0,0)=0)=1-\P(\omega(0,0)=-1)=p\in(0,1).
\end{align*}
For $\alpha>0$ let $P^\alpha$ denote the random walk with increment distribution 
\begin{align*}
P^\alpha (X_{t+1}=i|X_t=j)&=C(\alpha)e^{-|i-j|^{\alpha}},
\end{align*}
where $C(\alpha)$ is the normalizing constant. It is shown that there exists $\lambda_0(\alpha)\in(-\infty,0]$ such that almost surely
\begin{align*}
\lim_{t\to\infty}\tfrac 1t\E[\log Z^\alpha_t]=\lim_{t\to\infty}\tfrac 1t\log Z^\alpha_t=\lambda_0(\alpha).
\end{align*}
Observe that in this case large values of $\alpha$ mean that the increments concentrate mostly on $\{-1,0,1\}$ while for $\alpha\to 0$ the distribution becomes more spread out. One would thus expect that $\lambda_0(\alpha)$ is decreasing in $\alpha$.

\subsection{A comparison result for random walks on trees}\label{sec:trees}

In this section we show that $\preceq_M$ is a necessary and sufficient criterion for polymers on trees. We begin by defining the model: Let $V$ denote the $K$-ary tree. That is, $V$ is cycle-free and has a distinguished vertex $\o$ of degree $K$ while all other vertices have degree $K+1$. We say that $\o$ is the root and we write $|v|$ for the graph-distance between $v$ and $\o$. We call $|v|$ the height of $v$ and let $V_t$ denote for the set of vertices of height $t$. For $v\in V$ and $i\in\{1,\dots,K\}$ we let $(v,i)$ denote the $i^{th}$ descendant of $v$, and we write $D(v)$ for the set of descendants of $v$. 

A path $x$ in $V$ is a function $\N\to V$ such that $|x_t|=t$ and $x_{t+1}\in D(x_t)$ for every $t\in\N$. That is, a path moves away from the root in each step. Let $\Sigma$ denote the set of paths on $V$. It is an elementary observation that if $x, y\in\Sigma(V)$ satisfy $x_s\neq y_s$, then $x_t\neq y_t$ for all $t\geq s$. Clearly this is a special property of the tree which does not hold on the lattice. 

Let $\Omega\coloneqq [-1,\infty)^V$ denote the set of environments and define $F_t$ as in \eqref{eq:dtF} for the random polymer model on the lattice. Since there is no group structure on $\Sigma$ we have to redefine \eqref{ass:spatial}:
\begin{definition}
A \textbf{shift} is a bijection $\theta\colon V\to V$ such that if $v$ is a descendant of $w$ then $\theta(v)$ is a descendant of $\theta(w)$. 
\end{definition}
In words, shifts are bijections that respect the tree structure. 

\begin{assumption}\label{ass:iid}
$\P$ is invariant under all shifts $\theta$. That is for all $A\in\F$ and all shifts $\theta$ 
\begin{align*}
\P(\{\omega(v)\colon v\in V\}\in A)=\P(\{\omega(\theta(v))\colon v\in V\}\in A).
\end{align*}
\end{assumption}
Note that this assumption is satisfied for the canonical example of an i.i.d. environment. We consider a special class of shifts:
\begin{definition}\label{def:elementary}
A shift $\theta$ is called \textbf{elementary} if there exist $v\in V$ and a permutation $\pi$ of $\{1,\dots,K\}$ such that $\theta(w)=w$ if $w$ is not a (strict) descendant of $v$, and such that $\theta(w)\coloneqq (v,\pi(a),v')$ if $w=(v,a,v')$ is a descendant of $v$.
\end{definition}
In words, an elementary shift permutes all subtrees attached to the node $v$ according to the permutation $\pi$, and leaves everything else invariant. See Figure \ref{fig:tree} for an illustration. 
\begin{figure}
\includegraphics[width=\textwidth]{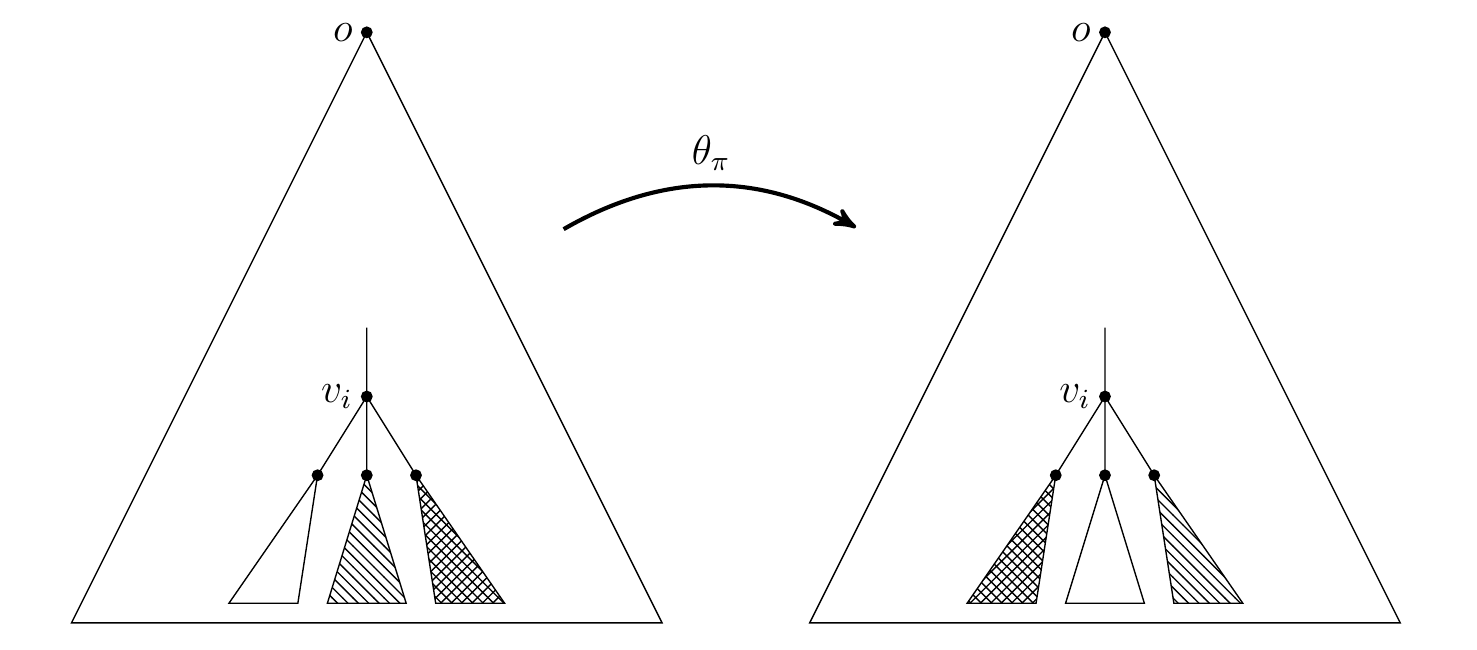}
\caption{In the case $K=3$ and $(\pi(1),\pi(2),\pi(3))=(2,3,1)$ the elementary shift $\theta$ associated to $v$ and $\pi$ permutes the subtrees attached to the descendants of $v$ while keeping the rest of the tree untouched.}\label{fig:tree}
\end{figure}

We finish the model by defining random walks on $V$: For $p\in\M(\{1,\dots,K\})$ let $P^p$ denote the law of the Markov chain with increment distribution $p$. More precisely $X_0=\o$ and for every $t\in\N,a\in\{1,\dots,K\}$
\begin{align*}
P^p(X_{t+1}=(v,a)|X_t=v)=p(a).
\end{align*}
In anticipation of the proof of Theorem \ref{thm:trees} we also introduce inhomogeneous random walks: If $p=\{p^{(v)}\colon v\in V\}$ is a collection of probability measures on $\{1,\dots,K\}$ we write $P^p$ for the Markov chain with transition probabilities
\begin{align*}
P^p(X_{t+1}=(v,a)|X_t=v)=p^{(v)}(a).
\end{align*}
We can now state the main result of this section:
\begin{theorem}\label{thm:trees}
Let $p,q\in\M(\{1,\dots,K\})$ and write $Z^p_t$ resp. $Z^q_t$ for the partition function of $P^p$ resp. $P^q$. Then $p\preceq_M q$ if and only if $Z^q_t\preceq_{cv}Z^p_t$ for all $\P$ satisfying \eqref{ass:inte} and \eqref{ass:iid}.
\end{theorem}

\begin{proof}[Proof of ``$\Leftarrow$'']
This is identical to the proof of Proposition \ref{prop:necc} since only the first step of the random walk is relevant in that proof.
\end{proof}
\begin{proof}[Proof of ``$\Rightarrow$'']
We show $Z_{t+1}^q\preceq_{cv}Z_{t+1}^p$. Let $\overline V_t$ denote the set of node of height at most $t$ and consider an enumeration $\overline V_t=\{v_1,v_2,\dots,v_N\}$ such that $i\mapsto |v_i|$ is decreasing. Note that this ensures that all descendants of $v_i$ in $\overline V_t$ are contained in $\{v_1,\dots,v_i\}$, for every $i=1,...,N$. Our aim is to incrementally transform $P^q$ into $P^p$ by considering a sequence of inhomogeneous random walks $P^{r_0},\dots,P^{r_N}$ such that $P^{r_0}=P^q$ and $P^{r_N}=P^p$, and such that from $P^{r_i}$ to $P^{r_{i+1}}$ we only change the increment distribution at one node. More precisely for $i=0,\dots,N$ and $v\in \overline V_t$ let
\begin{align*}
r_i^{(v_j)}\coloneqq \left\{\begin{matrix}p&\text{ if }j<i\\[1mm]q&\text{ else. }\end{matrix}\right.
\end{align*}
In particular $P^{r_0}=P^q$ and $P^{r_N}=P^q$. Let $W_i$ denote the partition function of $P^{r_i}$, and note that it is enough to show for all $i=0,\dots,N-1$ 
\begin{align*}
W_i\preceq_{cv} W_{i+1}.
\end{align*}
Let $x$ be a path that visits $v_i$, and note that if $F_{|v_i|}(\omega,x)=0$ then $F_t(\omega,y)=0$ for all $t\geq |v_i|$ and all paths visiting $v_i$. So on $\{F_{|v_i|}(\omega,x)=0\}$ there is nothing to prove since $W_i(\omega)=W_{i+1}(\omega)$. We therefore assume $F_{|v_i|}(\omega,x)>0$, and in this case we can write
\begin{align*}
W_i(\omega)&=A(\omega)+b(\omega)\sum_{a=1}^k q(a)\widehat W_i(a,\omega)\\
W_{i+1}(\omega)&=A(\omega)+b(\omega)\sum_{a=1}^k p(a)\widehat W_i(a,\omega),
\end{align*}
where
\begin{align*}
\widehat W_i(a,\omega)&\coloneqq E\,^{r_i}\Big[\frac{F_{t+1}(\omega,X)}{F_{|v_i|}(\omega,X)}\Big|X_{|v_i|+1}=(v_i,a)\Big]\\
A(\omega)&\coloneqq E^{r_i}\big[F_{t+1}(\omega,X)\1\{X_{|v_i|}\neq v_i\}\big]\\
b(\omega)&\coloneqq E^{r_i}\big[F_{|v_i|}(\omega,X)\1\{X_{|v_i|}=v_i\}\big].
\end{align*}


In words, $A(\omega)$ represents the contributions from paths that do not visit $v_i$ while in the second term we have the paths that visit $v_i$. This contribution can be split into the common part $b(\omega)$ (collected at the nodes $\o\to v_i$) and the remaining contribution $\widehat W_i(a,\omega)$ depending on which descendant of $v_i$ the path visits at time $|v_i|+1$. Note that this decomposition is only possible on a tree.

For $\pi$ a permutation of $\{1,\dots,K\}$ let $\theta^\pi$ denote the elementary shift (recall Definition \ref{def:elementary}) associated to $v_i$ and $\pi$. Let $\omega(\pi)$ be the environment obtained by
\begin{align*}
(\omega(\pi))(v)\coloneqq \omega(\theta^\pi(v)).
\end{align*}

Note that $\omega(\pi)$ has the same distribution as $\omega$ by \eqref{ass:iid}. Moreover by our choice for the enumeration $v_1,\dots,v_N$ we know that all descendants of $v_i$ have the same increment distribution, so that
\begin{align}\label{eq:shifted}
\widehat W_i(a,\omega(\pi))=\widehat W_i(\pi^{-1}(a),\omega).
\end{align}

Let $\S$ denote the set of permutations of $\{1,\dots,K\}$ and for $\pi\in\S$ define 
\begin{align*}
c(\pi,\omega)&\coloneqq \sum_{a=1}^K p(a)\widehat  W_i(a,\omega(\pi))\\
d(\pi,\omega)&\coloneqq \sum_{a=1}^K q(a)\widehat  W_i(a,\omega(\pi))
\end{align*}
In what follows we regard $c(\cdot,\omega)$ and $d(\cdot,\omega)$ as $K!$-dimensional real vectors.
\begin{claim}
For all $\omega$ we have $c(\cdot,\omega)\preceq_M d(\cdot,\omega)$.
\end{claim}
Let us first see how we can apply the claim: Let
\begin{align*}
C(\pi,\omega)&\coloneqq W_{i+1}(\omega(\pi))=A(\omega(\pi))+b(\omega(\pi))c(\omega,\pi)=A(\omega)+b(\omega)c(\pi,\omega),\\
D(\pi,\omega)&\coloneqq W_i(\omega(\pi))=A(\omega(\pi))+b(\omega(\pi))d(\omega,\pi)=A(\omega)+b(\omega)d(\pi,\omega).
\end{align*}
We have used that $\omega$ and $\omega(\pi)$ agree everywhere except at the descendants of $v_i$, so that $A$ and $b$ are invariant under the elementary shift associated to $v_i$ and $\pi$. So from the claim it is clear that 
\begin{align*}
C(\cdot,\omega)\preceq_M D(\cdot,\omega).
\end{align*}
Due to \cite[Theorem 3.C.1]{olkin} we get for all $f$ concave 
\begin{align}\label{eq:concavemajo}
\frac{1}{|\S|}\sum_{\pi\in S}f(C(\pi,\omega))\geq \frac{1}{|\S|}\sum_{\pi\in S}f(D(\pi,\omega)).
\end{align}
Let $\omega$ and $\Pi$ be independent, $\omega$ with distribution $\P$ and $\Pi$ chosen uniformly from $\S$. By \eqref{ass:iid} we see that $\omega(\Pi)$ has law $\P$ as well, and from \eqref{eq:concavemajo} 
\begin{align*}
\E[f(W_{i+1}(\omega))]
&=\E\Big[\frac{1}{|\S|}\sum_{\pi\in \S}f(W_{i+1}(\omega(\pi)))\Big]\\
&\geq\E\Big[\frac{1}{|\S|}\sum_{\pi\in \S}f(W_{i}(\omega(\pi)))\Big]
=\E[f(W_i(\omega))].
\end{align*}
\end{proof}

\begin{proof}[Proof of the claim]
Consider the matrix $M\in[0,\infty)^{\{1,\dots,K\}\times \S}$ defined by
\begin{align*}
M(a,\pi)\coloneqq \widehat  W_i(a,\omega(\pi)).
\end{align*}
Clearly that $c=p M$ and $d=q M$, where we interpret $p$ and $q$ as $K$-dimensional row vectors. The claim then follows from \cite[Proposition 5.A.17]{olkin} after verifying that the matrix $M$ has the correct form. But this is easily done with the help of  \cite[Proposition 5.A.17a]{olkin}, which states that it is enough to check the following: If $m=(m_1,\dots,m_K)$ is a column of $M$ and $\sigma\in\S$ then
\begin{align*}
\widehat  m\coloneqq (m_{\sigma(1)},\dots,m_{\sigma(K)})
\end{align*}
is also a column of $M$. To see this let $m$ be the column corresponding to $\pi\in S$, i.e. $m_a=\widehat W_i(a,\omega(\pi))$. Then from \eqref{eq:shifted}
\begin{align*}
(\widehat m)_a=m_{\sigma(a)}=\widehat W_i(\sigma(a),\omega(\pi))=\widehat W_i(a,\omega(\sigma^{-1}\circ\pi)).
\end{align*}
Thus $\widehat m$ is the column corresponding to $\sigma^{-1}\circ \pi\in\S$.
\end{proof}

\section*{Acknowledgments}
The author would like to thank Noam Berger, David Criens, Lexuri Fernandez and Nina Gantert for carefully proofreading the manuscript and for many helpful comments.

%% file: arXiv_template.bbl
\begin{thebibliography}{10}

\bibitem{carmona2}
Ren\'{e}~A. Carmona and Stanislav Molchanov.
\newblock Stationary parabolic {A}nderson model and intermittency.
\newblock {\em Probab. Theory Related Fields}, 102(4):433--453, 1995.

\bibitem{comets_survey}
Francis Comets.
\newblock {\em Directed polymers in random environments}, volume 2175 of {\em
  Lecture Notes in Mathematics}.
\newblock Springer, Cham, 2017.
\newblock Lecture notes from the 46th Probability Summer School held in
  Saint-Flour, 2016.

\bibitem{discretedisasters}
Francis Comets, Ryoki Fukushima, Shuta Nakajima, and Nobuo Yoshida.
\newblock Limiting results for the free energy of directed polymers in random
  environment with unbounded jumps.
\newblock {\em J. Stat. Phys.}, 161(3):577--597, 2015.

\bibitem{yoshida_path}
Francis Comets, Tokuzo Shiga, and Nobuo Yoshida.
\newblock Directed polymers in a random environment: path localization and
  strong disorder.
\newblock {\em Bernoulli}, 9(4):705--723, 2003.

\bibitem{yoshida_survey}
Francis Comets, Tokuzo Shiga, and Nobuo Yoshida.
\newblock Probabilistic analysis of directed polymers in a random environment:
  a review.
\newblock In {\em Stochastic analysis on large scale interacting systems},
  volume~39 of {\em Adv. Stud. Pure Math.}, pages 115--142. Math. Soc. Japan,
  Tokyo, 2004.

\bibitem{comets2004some}
Francis Comets and Nobuo Yoshida.
\newblock Brownian directed polymers in random environment.
\newblock {\em Comm. Math. Phys.}, 254(2):257--287, 2005.

\bibitem{yoshida}
Francis Comets and Nobuo Yoshida.
\newblock Branching random walks in space-time random environment: survival
  probability, global and local growth rates.
\newblock {\em J. Theoret. Probab.}, 24(3):657--687, 2011.

\bibitem{cranston}
Michael Cranston and Thomas Mountford.
\newblock Lyapunov exponent for the parabolic {A}nderson model in {$\bold
  R^d$}.
\newblock {\em J. Funct. Anal.}, 236(1):78--119, 2006.

\bibitem{hollander_survey}
Frank den Hollander.
\newblock {\em Random polymers}, volume 1974 of {\em Lecture Notes in
  Mathematics}.
\newblock Springer-Verlag, Berlin, 2009.
\newblock Lectures from the 37th Probability Summer School held in Saint-Flour,
  2007.

\bibitem{drewitzsurvival}
Alexander Drewitz, J\"{u}rgen G\"{a}rtner, Alejandro~F. Ram\'{i}rez, and
  Rongfeng Sun.
\newblock Survival probability of a random walk among a {P}oisson system of
  moving traps.
\newblock In {\em Probability in complex physical systems}, volume~11 of {\em
  Springer Proc. Math.}, pages 119--158. Springer, Heidelberg, 2012.

\bibitem{erhardparabolic}
Dirk Erhard, Frank den Hollander, and Gr\'{e}gory Maillard.
\newblock The parabolic {A}nderson model in a dynamic random environment: basic
  properties of the quenched {L}yapunov exponent.
\newblock {\em Ann. Inst. Henri Poincar\'{e} Probab. Stat.}, 50(4):1231--1275,
  2014.

\bibitem{fukushima}
Ryoki Fukushima and Stefan Junk.
\newblock Zero temperature limit for the brownian directed polymer among
  poissonian disasters.
\newblock {\em arXiv preprint arXiv:1810.09600}, 2018.

\bibitem{shiga2}
Tasuku Furuoya and Tokuzo Shiga.
\newblock Sample {L}yapunov exponent for a class of linear {M}arkovian systems
  over {$\bold Z^d$}.
\newblock {\em Osaka J. Math.}, 35(1):35--72, 1998.

\bibitem{self}
Nina Gantert and Stefan Junk.
\newblock A branching random walk among disasters.
\newblock {\em Electron. J. Probab.}, 22:Paper No. 67, 34, 2017.

\bibitem{garetmarchand}
Olivier Garet and R\'{e}gine Marchand.
\newblock The critical branching random walk in a random environment dies out.
\newblock {\em Electron. Commun. Probab.}, 18:no. 9, 15, 2013.

\bibitem{gaertnerquenched}
J\"{u}rgen G\"{a}rtner, Frank den Hollander, and Gr\'{e}gory Maillard.
\newblock Quenched {L}yapunov exponent for the parabolic {A}nderson model in a
  dynamic random environment.
\newblock In {\em Probability in complex physical systems}, volume~11 of {\em
  Springer Proc. Math.}, pages 159--193. Springer, Heidelberg, 2012.

\bibitem{koenigparabolic}
Wolfgang K\"{o}nig.
\newblock {\em The parabolic {A}nderson model}.
\newblock Pathways in Mathematics. Birkh\"{a}user/Springer, [Cham], 2016.
\newblock Random walk in random potential.

\bibitem{olkin}
Albert~W. Marshall, Ingram Olkin, and Barry~C. Arnold.
\newblock {\em Inequalities: theory of majorization and its applications}.
\newblock Springer Series in Statistics. Springer, New York, second edition,
  2011.

\bibitem{mullerstoyan}
Alfred M\"{u}ller and Dietrich Stoyan.
\newblock {\em Comparison methods for stochastic models and risks}.
\newblock Wiley Series in Probability and Statistics. John Wiley \& Sons, Ltd.,
  Chichester, 2002.

\bibitem{nakajima}
Shuta Nakajima.
\newblock Concentration results for directed polymer with unbounded jumps.
\newblock {\em ALEA Lat. Am. J. Probab. Math. Stat.}, 15(1):1--20, 2018.

\bibitem{nguyen}
Vu~Lan Nguyen.
\newblock A note about domination and monotonicity in disordered systems.
\newblock {\em arXiv preprint arXiv:1606.01835}, 2016.

\bibitem{shiga}
Tokuzo Shiga.
\newblock Exponential decay rate of survival probability in a disastrous random
  environment.
\newblock {\em Probab. Theory Related Fields}, 108(3):417--439, 1997.

\bibitem{yoshida_CLT}
Nobuo Yoshida.
\newblock Central limit theorem for branching random walks in random
  environment.
\newblock {\em Ann. Appl. Probab.}, 18(4):1619--1635, 2008.

\end{thebibliography}
